\documentclass[12pt]{amsart}
\usepackage{amsmath, amsthm, amscd, amsfonts, amssymb, graphicx, color}
\usepackage[pagebackref=false,colorlinks,linkcolor=blue,citecolor=red]{hyperref}
\usepackage{accents}
\usepackage[all]{xy}

\usepackage{pgf,tikz,float}
\usepackage{mathrsfs}
\usetikzlibrary{arrows}
\def\NZQ{\mathbb}               % the font for N,Z,Q,R,C
\def\NN{{\NZQ N}}
\def\QQ{{\NZQ Q}}

\def\frk{\frak}               % font for "Fraktur"

\def\Phi{{\frk n}}
\def\Phi{{\frk N}}
\def\KK{{\mathbb K}}
\def\opn#1#2{\def#1{\operatorname{#2}}} % to make operators
\opn\chara{char}
\opn\length{\ell}
\opn\pd{pd}
\opn\rk{rk}
\opn\projdim{proj\,dim}
\opn\cohdim{cd}
\opn\injdim{inj\,dim}
\opn\rank{rank}
\opn\depth{depth}
\opn\diam{diam}
\opn\grade{grade}
\opn\height{height}
\opn\embdim{emb\,dim}
\opn\codim{codim}

\opn\Tr{Tr}
\opn\bigrank{big\,rank}
\opn\superheight{superheight}
\opn\lcm{lcm}
\opn\trdeg{tr\,deg}
\opn\reg{reg}
\opn\hilb{Hilb}
\opn\hpolynomial{h}
\opn\cdeg{cdeg}
\opn\lreg{lreg}
\opn\ini{in}
\opn\lpd{lpd}
\opn\size{size}
\opn\bigsize{bigsize}
\opn\cosize{cosize}
\opn\bigcosize{bigcosize}
\opn\sdepth{sdepth}
\opn\sreg{sreg}
\opn\link{link}
\opn\fdepth{fdepth}
\opn\lin{lin}
\opn\ini{in}
\opn\div{div}
\opn\Div{Div}
\opn\cl{cl}
\opn\Cl{Cl}
\opn\Spec{Spec}
\opn\Supp{Supp}
\opn\supp{supp}
\opn\Sing{Sing}
\opn\Ass{Ass}
\opn\Min{Min}
\opn\Mon{Mon}
\opn\dstab{dstab}
\opn\astab{astab}
\opn\Syz{Syz}
\opn\Ann{Ann}
\opn\Rad{Rad}
\opn\Soc{Soc}
\opn\Im{Im}
\opn\Ker{Ker}
\opn\Coker{Coker}
\opn\Am{Am}
\opn\Hom{Hom}
\opn\Tor{Tor}
\opn\Ext{Ext}
\opn\End{End}
\opn\Aut{Aut}
\opn\id{id}

\opn\nat{nat}
\opn\pff{pf}%   \pf existss already
\opn\Pf{Pf}
\opn\GL{GL}
\opn\SL{SL}
\opn\mod{mod}
\opn\ord{ord}
\opn\Gin{Gin}
\opn\Hilb{Hilb}
\opn\sort{sort}
\opn\initial{init}
\opn\ende{end}
\opn\height{height}
\opn\type{type}
\opn\mdeg{mdeg}
\opn\aff{aff}
\opn\con{conv}
\opn\relint{relint}
\opn\st{st}
\opn\lk{lk}
\opn\cn{cn}
\opn\core{core}
\opn\vol{vol}
\opn\link{link}
\opn\lex{lex}
\opn\sign{sign}
\opn\gr{gr}

\def\pot#1#2{#1[\kern-0.28ex[#2]\kern-0.28ex]}

\opn\dirlim{\underrightarrow{\lim}}
\opn\inivlim{\underleftarrow{\lim}}
\def\Implies{\ifmmode\Longrightarrow \else
	\unskip${}\Longrightarrow{}$\ignorespaces\fi}
\def\implies{\ifmmode\Rightarrow \else
	\unskip${}\Rightarrow{}$\ignorespaces\fi}
\def\iff{\ifmmode\Longleftrightarrow \else
	\unskip${}\Longleftrightarrow{}$\ignorespaces\fi}

\let\:=\colon
\newtheorem{Theorem}{Theorem}[section]
\newtheorem{Lemma}[Theorem]{Lemma}
\newtheorem{Corollary}[Theorem]{Corollary}

\newtheorem{Remark}[Theorem]{Remark}

\newtheorem{Definition}[Theorem]{Definition}

\let\epsilon\varepsilon
\def\pnt{{\raise0.5mm\hbox{\large\bf.}}}

\begin{document}
	\title{On the depth of binomial edge ideals of graphs}
	\author {M. Rouzbahani Malayeri, S. Saeedi Madani, D. Kiani}
		
	\address{Mohammad Rouzbahani Malayeri, Department of Mathematics and Computer Science, Amirkabir University of Technology (Tehran Polytechnic), Tehran, Iran}
	\email{m.malayeri@aut.ac.ir}
	
		\address{Sara Saeedi Madani, Department of Mathematics and Computer Science, Amirkabir University of Technology (Tehran Polytechnic), Tehran, Iran, and School of Mathematics, Institute for Research in Fundamental Sciences (IPM), Tehran, Iran}
	\email{sarasaeedi@aut.ac.ir}
	
		\address{Dariush Kiani, Department of Mathematics and Computer Science, Amirkabir University of Technology (Tehran Polytechnic), Tehran, Iran}
	\email{dkiani@aut.ac.ir}

	\begin{abstract}
Let $G$ be a graph on the vertex set $[n]$ and $J_G$ the associated binomial edge ideal in the polynomial ring $S=\KK[x_1,\ldots,x_n,y_1,\ldots,y_n]$. In this paper we investigate the depth of binomial edge ideals. More precisely, we first establish a combinatorial lower bound for the depth of $S/J_G$ based on some graphical invariants of $G$. Next, we combinatorially characterize all binomial edge ideals $J_G$ with $\depth S/J_G=5$. To achieve this goal, we associate a new poset $\mathcal{M}_G$ with the binomial edge ideal of $G$, and then elaborate some topological properties of certain subposets of $\mathcal{M}_G$ in order to compute some local cohomology modules of~$S/J_G$.

	\end{abstract}

	%\thanks{$^\ast$ Corresponding author}
	%\thanks{}
	
	\subjclass[2010]{13C15; 05E40; 13C70; 13C05}
	\keywords{Binomial edge ideals, depth, diameter of a 
	 graph, Hochster type formula, meet-contractible}
	
	\maketitle
	
\section{Introduction}\label{intro}

Over the last two decades, the study of ideals with combinatorial origins has been an appealing trend in commutative algebra. One of the most well-studied types of such ideals which has attracted special attention in the literature is the binomial edge ideal of a graph. 
\par Let $G$ be a graph on $[n]$ and $S=\KK[x_1 ,\ldots ,x_n , y_1 , \ldots , y_n]$ be the polynomial ring over a field $\KK$. Then, the \emph{binomial edge ideal} associated with $G$, denoted by $J_G$, is the ideal in $S$ generated by all the quadratic binomials of the form $f_{ij}=x_{i}y_{j}-x_{j}y_{i}$, where $\{i,j\}\in E(G)$ and $1\leq i<j\leq n$. This class of ideals was  introduced in 2010 by Herzog, Hibi, Hreinsd{\'o}ttir, Kahle and Rauh in \cite{HHHKR}, and independently by Ohtani in \cite{O}, as a natural generalization of determinantal ideals, as well as ideals  generated by adjacent $2$-minors of a $2\times n$-matrix of indeterminates. 
\par Since then, many researchers have studied the algebraic properties and homological invariants of binomial edge ideals. The main goal is to understand how the invariants of the associated graph are reflected in the algebra of the ideal and vice versa. Indeed, it has been proved that there exists a mutual interaction between algebraic properties of binomial edge ideals and  combinatorial properties of the underlying graphs, see e.g.,  \cite{A, BN, BMS, BMS2, EHH, ERT, ERT2, EZ, G, JKS, KS2, Kumar3, MM, RSK2, RSK3, SK, SK2} for some efforts in this direction.
\par One of the homological invariants associated with  binomial edge ideals, which is not easy to compute, is \emph{depth}. Recall that 
\[
\depth S/J_G=\min \{i:H_{\mathfrak{m}}^{i}(S/J_G)\neq 0\},
\]
where $H_{\mathfrak{m}}^{i}(S/J_G)$ denotes the $i^{th}$ local cohomology module of $S/J_G$ with support at the maximal homogeneous ideal $\mathfrak{m}=(x_1,\ldots,x_n,y_1,\ldots,y_n)$ of $S$. 
\par Notice that unlike some other homological invariants associated with binomial edge ideals, like the Castelnuovo-Mumford regularity, which has been studied extensively, little is known about the depth of binomial edge ideals. Moreover, since the  depth is in general dependent on the characteristic of the base field, finding some characteristic-free results about the depth of such homogeneous ideals is of great interest. In the following we briefly summarize the results in this direction.
\par The first important result about the depth of binomial edge ideals appeared in~\cite{EHH}, where the authors showed that $\depth S/J_G=n+1$, where $G$ is a connected \emph{block} graph. Afterward, in \cite{ZZ} it was shown that $\depth S/J_G=n$, where $G$ is a cycle with more than $3$ vertices. Later, this result was generalized in~\cite{Kumar2} for the so-called \emph{quasi-cycle} graphs. Also, it was shown in~\cite{S} that the depth of the binomial edge ideal of a \emph{unicyclic} graph is either $n$ or $n+1$, and all unicyclic graphs of depth $n$ and $n+1$ were characterized. In~\cite{KumarS}, by generalizing a formula for the depth of \emph{complete bipartite} graphs in~\cite{SZ}, the authors gave some precise formulas for the depth of the \emph{join product} of graphs.
In \cite{BN}, Banerjee and N\'u\~nez-Betancourt established a nice combinatorial upper bound for the depth of binomial edge ideals in terms of the \emph{vertex connectivity} of the underlying graph. Indeed, for a non-complete connected graph $G$, they showed that
\begin{equation}\label{upper}
\depth S/J_G\leq n-\kappa(G)+2,
\end{equation} 
where $\kappa(G)$ denotes the vertex connectivity of $G$, which is the minimum number of vertices of $G$ whose removal makes $G$ disconnected.
\par It is worth mentioning here that to the best of our knowledge, besides the above combinatorial upper bound for the depth of binomial edge ideals, there is no combinatorial lower bound for the depth of $S/J_G$. So, as the first main result of this paper, we supply the following combinatorial lower bound for the depth of binomial edge ideals:

\medskip

\noindent\textbf{Theorem A}~(Theorem \ref{mixed})\textbf{.}
\emph{Let $G$ be a graph on $[n]$. Then
$$\depth S/J_G\geq \xi(G).$$
In particular, if $G$ is connected, then 
$$\depth S/J_G\geq f(G)+\diam(G).$$
}
\par Here, $\xi(G)=f(G)+d(G)$ where $f(G)$ denotes the number of \emph{free} vertices (or \emph{simplicial} vertices) of the graph $G$, and $d(G)$ denotes the sum of the \emph{diameters} of the connected components of $G$, and the number of the isolated vertices of $G$. Moreover, $\diam(G)$ denotes the diameter of $G$. 
\par In order to prove Theorem A, we first introduce the concept of $d$-compatible maps. Such maps are defined from the set of all graphs to the set of non-negative integers that admit certain properties (see Definition \ref{psi}). Then, we exploit this concept to establish a general lower bound for the depth of binomial edge ideals. We also provide a combinatorial $d$-compatible map, which yields the above lower bound. We also show that our lower bound is best possible in the sense that there are graphs $G$ for which $\depth S/J_G=\xi(G)$, see Figures \ref{sharp} and \ref{new}.

\par Now, we would like to mention another motivation of this paper.  In \cite{RSK2}, the authors of the present paper provided a general lower bound for the depth of binomial edge ideals. Indeed, they showed that $\depth S/J_G\geq 4$, where $G$ is a connected graph with at least three vertices.
Moreover, they gave an explicit characterization of the graphs $G$ for which $\depth S/J_G=4$. More precisely, they showed that for graphs $G$ with more than three vertices, $\depth S/J_G=4$ if and only if $G=G'\ast 2K_1$, for some graph $G'$, where $G'\ast 2K_1$ denotes the join product of a graph $G'$ and two isolated vertices denoted by $2K_1$. 
\par Now, it is natural to ask about a combinatorial characterization of binomial edge ideals of higher depths. In this paper, by using a wonderful Hochster type formula for the local cohomology modules of binomial edge ideals provided by \`Alvarez~ Montaner in \cite{A}, we give such a characterization. Indeed, we prove the following characterization of the graphs $G$ for which $\depth S/J_G=5$, see Definition \ref{D5} for the required notation.

\medskip

\noindent\textbf{Theorem B}~(Theorem \ref{main})\textbf{.}
\emph{Let $G$ be a graph on $[n]$ with $n\geq 5$. Then the following statements are equivalent:}
\begin{enumerate}
\item[{(a)}] $\depth S/J_G=5$.
\item[{(b)}] $G$ is a $D_5$-type graph.
\end{enumerate}
\par \medskip The proof of Theorem B involves some topological results that we obtain in this paper about some specific subposets of a poset associated with  binomial edge ideals, which are indeed of independent interest.
\par The organization of this paper is as follows. In Section \ref{per}, we fix the notation and review some facts and  definitions that will be used throughout the paper.
\par In Section \ref{benefit}, toward providing some lower bounds for the depth of binomial edge ideals, in Definition \ref{psi}, we introduce a concept which is named as $d$-compatible maps. Such maps are defined from the set of all graphs to the set of non-negative integers with some specific properties. Then, in Theorem \ref{dcompatible}, by using the aforementioned concept,  a general lower bound is given for the depth of binomial edge ideals. In addition, after providing a combinatorial $d$-compatible map in Theorem~\ref{xi}, a combinatorial lower bound is given for the depth of binomial edge ideals in Theorem~\ref{mixed}. This bound,  together with a result  from \cite{RSK2}, provides a modified version of the bound given in Theorem \ref{mixed}.
\par In Section \ref{hochster type}, following the poset theoretical as well as the topological approaches used in \cite{A} and \cite{RSK2}, we associate a new poset with binomial edge ideals in Definition \ref{typo}. Then, we state in Theorem~\ref{Hochster}, the Hochster type formula for the local cohomology modules of binomial edge ideals arose from \cite[Theorem~3.9]{A}.   
\par  In Section \ref{topology}, we use the Hochster type formula provided in Section \ref{hochster type} to characterize all graphs $G$ for which $\depth S/J_G=5$, in Theorem~\ref{main}. To prove our characterization, we need to provide several auxiliary ingredients. In particular, Theorem \ref{corona} which studies the vanishing of the zeroth and the first reduced cohomology groups of some subposets of the associated poset with    binomial edge ideals, plays a crucial role in the proof of Theorem \ref{main}.

\section{Preliminaries}\label{per}
In this section, we recall some notions and known facts that are used in this paper.

\par \medskip

\textbf{Graph theory.} Throughout the paper, all graphs are assumed to be simple (i.e. with no loops, directed and multiple edges). Let $G$ be a graph on the vertex set $[n]$ and $T\subseteq [n]$. A subgraph $H$ of $G$ on the vertex set $T$ is called an \emph{induced subgraph} of $G$, whenever for any two vertices $u,v\in T$, one has $\{u,v\}\in E(H)$ if $\{u,v\}\in E(G)$. Now, by $G-T$, we mean the induced subgraph of $G$ on the vertex set $[n]\backslash T$. In the special case, when $T=\{v\}$, we use the notation $G-v$ instead of $G-\{v\}$, for simplicity.
 
\par Let $v\in [n]$. Denoted by $N_G(v)$, is the set of all vertices of $G$ which are adjacent to $v$. We say that $v$ is a   \emph{free} vertex (or \emph{simplicial} vertex) of $G$, if the induced subgraph of $G$ on the vertex set $N_G(v)$ is a complete graph. Moreover, a vertex which is not free, is called a non-free (or non-simplicial) vertex. We use $f(G)$ and $iv(G)$ to denote the number of free vertices and the number of non-free vertices of a graph $G$, respectively.

\par A vertex $v\in [n]$ is said to be a \emph{cut vertex} of $G$ whenever $G-v$ has more connected components than $G$. Moreover, we say that $T$ has the  \emph{cut point property} for $G$, whenever each $v\in T$ is a cut vertex of the graph $G-(T\backslash \{v\})$. Particularly,  the empty set $\emptyset$, has the cut point property for $G$.
\par Let $G_1$ and $G_2$ be two graphs on the disjoint vertex sets $V(G_1)$ and $V(G_2)$, respectively. Then, by the \emph{join product} of $G_1$ and $G_2$, denoted by $G_1*G_2$, we mean the graph on the vertex set $V(G_1)\cup V(G_2)$ and with the edge set 
\[
E(G_1)\cup E(G_2)\cup \{\{u,v\}: u\in V(G_1)~\mathrm{and}~v\in V(G_2)\}.
\]
\par
\textbf{Primary decomposition of binomial edge ideals.} Let $G$ be a graph on $[n]$ and $T\subseteq [n]$. Let also $G_1 , \ldots , G_{c_G(T)}$ be the connected components of $G-T$, and $\widetilde{G}_1 , \ldots , \widetilde{G}_{c_G(T)}$ be the complete graphs on the vertex sets $V(G_1), \ldots , V(G_{c_G(T)})$, respectively. Let
\[
P_T(G)=(x_v,y_v)_{v\in T} +J_{\widetilde{G}_1}+\cdots +J_{\widetilde{G}_{c_G(T)}}.
\]
Then, by \cite[Theorem~3.2]{HHHKR}, it is known that $J_{G}=\bigcap\limits_{\substack{T\subseteq [n]}}P_T(G)$. Moreover,  in \cite[Corollary~3.9]{HHHKR}, all the minimal prime ideals of $J_G$ were determined. Indeed, it was shown that $P_T(G)\in \Min(J_G)$ if and only if $T\in \mathcal{C}(G)$, where 
\[
\mathcal{C}(G)=\{T\subseteq [n]:T~\mathrm{has~the~cut~point~property~for}~G\}.
\]
Finally, the following useful formula could be easily verified:
\[
\mathrm{ht}\hspace{0.9mm}P_T(G)=n-c_G(T)+|T|.
\]
\par
\textbf{Poset topology.} Let $\Delta$ be a simplicial complex. Then, by the \emph{$1$-skeleton} graph of $\Delta$ we mean the subcomplex of $\Delta$ consisting of all the faces of $\Delta$ which have cardinality at most $2$. The simplicial complex $\Delta$ is said to be \emph{connected} if its $1$-skeleton graph is connected.
\par  Let $(\mathcal{P},\preccurlyeq)$ be a poset. Recall that the \emph{order complex} of $\mathcal{P}$, denoted by $\Delta(\mathcal{P})$, is the simplicial complex whose facets are the maximal chains in $\mathcal{P}$. We say that $\mathcal{P}$ is a  connected poset if its order complex $\Delta(\mathcal{P})$ is connected. Similarly, we say that $\mathcal{P}$ is contractible if $\Delta(\mathcal{P})$ is contractible.    If $\mathcal{P}$ is an empty poset, then we consider $\Delta(\mathcal{P})=\{\emptyset\}$, i.e. the empty simplicial complex.
\par \medskip 
\textbf{Mayer-Vietoris sequence.} Let $\Delta$ be a simplicial complex and $v\in V(\Delta)$. 
Recall the following three subcomplexes of $\Delta$ that will be used in this paper.
\begin{itemize}

\item $\mathrm{star}_\Delta(v)=\{\sigma\in\Delta : \sigma\cup \{v\}\in \Delta\}$;

\item $\mathrm{del}_\Delta(v)=\{\sigma\in\Delta : v\not\in\sigma\}$;
\item $\mathrm{link}_\Delta(v)=\{\sigma\in\Delta : v\not\in\sigma \hbox{ and } \sigma\cup \{v\}\in \Delta\}$.
\end{itemize}
\par Let $\Delta_1=\mathrm{star}_\Delta(v)$ and $\Delta_2=\mathrm{del}_\Delta(v)$. Then $\Delta_1\cup \Delta_2=\Delta$ and $\Delta_1\cap \Delta_2=\mathrm{link}_\Delta(v)$, and we have the \emph{Mayer-Vietoris} sequence:
\begin{eqnarray*}
 \cdots \rightarrow H_{i}(\mathrm{link}_\Delta(v);\KK) \rightarrow  H_i(\mathrm{star}_\Delta(v);\KK)\oplus H_i(\mathrm{del}_\Delta(v);\KK)\rightarrow\\
H_i(\Delta;\KK)\rightarrow H_{i-1}(\mathrm{link}_\Delta(v);\KK)\rightarrow\cdots \hspace{1cm}
\end{eqnarray*} 
Moreover, we have the reduced version
\[
\cdots \rightarrow \widetilde{H}_{0}(\mathrm{link}_\Delta(v);\KK)\rightarrow \widetilde{H}_0(\mathrm{star}_\Delta(v);\KK)\oplus \widetilde{H}_0(\mathrm{del}_\Delta(v);\KK)\rightarrow \widetilde{H}_0(\Delta;\KK)\rightarrow 0,
\] 
provided that $\Delta_1\cap\Delta_2\neq \{\emptyset\}$.

\section{A combinatorial lower bound for the depth of binomial edge ideals}\label{benefit}
Our main goal in this section is to establish some lower bounds for the depth of binomial edge ideals. We first introduce the concept of $d$-\emph{compatible} maps,  which are defined from the set of all graphs to the set of non-negative integers $\NN_0$ with some desirable properties. Then, considering this concept, we give a general lower bound for the depth of binomial edge ideals. We also provide a combinatorial $d$-compatible map to obtain a combinatorial lower bound for the depth of such ideals as well.

\par We first introduce a graph that plays an important role in proving the main theorem of this section. Let $G$ be a graph on $[n]$ and $v\in [n]$. Associated with the vertex $v$, there is a graph, denoted by $G_{v}$, with the vertex set $V(G)$ and the edge set 
\[
E(G)\cup \{\{u,w\}: \{u,w\}\subseteq N_{G}(v)\}.
\]
Note that by the definition, it is clear that $v$ is a free vertex of the graph $G_v$, and $N_G(v)=N_{G_v}(v)$.
\par Now we are ready to define the notion of a $d$-compatible map. In the following, by $K_t$ we mean the complete graph on $t$ vertices, for every $t\in \mathbb{N}$.  
\begin{Definition}\label{psi}
\em{
Let $\mathcal{G}$ be the set of all graphs. A map $\psi:\mathcal{G}\longrightarrow \NN_0$ is called \emph{d-compatible}, if it satisfies the following conditions:
\begin{enumerate}
\item[{(a)}] if $G=\dot{\cup}_{i=1}^{t}K_{n_i}$, where $n_i\geq 1$ for every $1\leq i\leq t$, then $\psi(G)\leq t+\sum_{i=1}^{t}n_i$;

\item[{(b)}] if $G\neq \dot{\cup}_{i=1}^{t}K_{n_i}$, then there exists a non-free vertex $v\in V(G)$ such that
\begin{enumerate}
\item[{(1)}] $\psi(G-v)\geq \psi(G)$, and
\item[{(2)}] $\psi(G_v)\geq \psi(G)$, and
\item[{(3)}] $\psi(G_{v}-v)\geq \psi(G)-1$.
\end{enumerate}
\end{enumerate}
}
\end{Definition}

\par We also use the following lemma from \cite{Kumar}. In the following, $iv(G)$ denotes the number of non-free vertices of a graph $G$.
  
\begin{Lemma}{\cite[Lemma~3.4]{Kumar}}\label{simplicial}
Let $G$ be a graph and $v$ be a non-free vertex of $G$. Then, 
$\max\{iv(G_v), iv(G-v), iv(G_{v}-v)\}< iv(G)$.
\end{Lemma}

\par \medskip
The following theorem provides a general lower bound for the depth of binomial edge ideals. 
 
\begin{Theorem}\label{dcompatible}
Let $G$ be a graph on $[n]$ and $\psi$ a $d$-compatible map. Then 
$$\depth S/J_G\geq \psi(G).$$
\end{Theorem}
\begin{proof}
We prove the assertion by using induction on $iv(G)$. If $iv(G)=0$, then $G$ is a disjoint union of complete graphs, that is,  $G=\dot{\cup}_{i=1}^{t}K_{n_i}$, where $n_i\geq 1$ for every $1\leq i\leq t$. We have $\depth S/J_G=t+\sum_{i=1}^{t}n_i$, by \cite[Theorem~1.1]{EHH}. On the other hand, by condition $(a)$ of Definition \ref{psi} we have $\psi(G)\leq t+\sum_{i=1}^{t}n_i$, so that the assertion holds in this case. Now, we  assume that $iv(G)>0$. Let $v\in [n]$ be a non-free vertex with the properties of condition $(b)$ of Definition~\ref{psi}. 
\par By \cite[Lemma~4.8]{O}, we have $J_G=J_{G_v}\cap((x_v,y_v) + J_{G-v})$. Therefore, the short exact sequence
\[
0\longrightarrow  \dfrac{S}{J_G}\longrightarrow \dfrac{S}{J_{G_v}}\oplus \dfrac{S_v}{J_{G-v}} \longrightarrow  \dfrac{S_v}{J_{{G_v}-v}}\longrightarrow 0
\]
is induced, where $S_v=\KK[x_i,y_i:i\in[n]\backslash \{v\}]$. 
\par Now, the well-known depth lemma implies that
\begin{equation}\label{exact}
\depth S/J_G\geq \min\{\depth S/J_{G_v}, \depth S_v/J_{G-v}, \depth S_v/J_{G_v-v}+1\}.
\end{equation}
Moreover, by Lemma \ref{simplicial}, induction hypothesis and by Definition~\ref{psi} part $(b)$, we have
\begin{equation}\label{first}
\depth S/J_{G_v}\geq \psi(G_v)\geq \psi(G),
\end{equation}
\begin{equation}\label{second}
\hspace{0.04cm}\depth S_v/J_{G-v}\geq \psi(G-v)\geq \psi(G)
\end{equation}
and
\begin{equation}\label{third}
\hspace{0.9cm}\depth S_v/J_{G_v-v}\geq \psi(G_{v}-v)\geq \psi(G)-1.
\end{equation}
So, \eqref{exact} together with \eqref{first}, \eqref{second},   and \eqref{third} imply the result.
\end{proof} 

%\par \medskip
 
Now, we are going to provide a combinatorial $d$-compatible map. Before that, we need to recall the concept of \emph{diameter} of a connected graph. Let $G$ be a connected graph on $[n]$, $u$ and $v$ be two vertices of $G$. Then, by the \emph{distance} between the vertices $u$ and $v$ in $G$, which is denoted by $d_G(u,v)$, we mean the length of a shortest path connecting $u$ and $v$ in $G$. Now, the \emph{diameter} of $G$, denoted by $\diam(G)$, is defined as
$$\diam(G)=\max\{d_G(u,v): u,v \in V(G)\}.$$
We call a shortest path between two vertices $u$ and $v$ of $G$  with $d_G(u,v)=\diam(G)$, an \emph{LSP} in $G$.
\par Let $G$ be a graph on $[n]$ with the connected components $G_1,\ldots,G_t$. Then we set 
$$d(G):=\mathfrak{i}(G)+\sum_{i=1}^{t}\diam(G_i),$$
where $\mathfrak{i}(G)$ denotes the number of isolated vertices of $G$.

\par Now, in the next theorem, we provide a $d$-compatible map given by $\xi(G)$. Here, $f(G)$ denotes the number of free vertices of $G$.

\begin{Theorem}\label{xi}
The map $\xi:\mathcal{G}\longrightarrow \NN_0$ defined by
$$\xi(G)=f(G)+d(G),~\text{for~every}~G\in\mathcal{G}$$
is $d$-compatible.
\end{Theorem}
\begin{proof}
Let $G\in \mathcal{G}$. First assume that $G=\dot{\cup}_{i=1}^{t}K_{n_i}$, where $n_i\geq 1$ for every $1\leq i\leq t$. Clearly we have $f(G)= \sum_{i=1}^{t}n_i$ and $d(G)=t$, and hence $\xi(G)=t+\sum_{i=1}^{t}n_i$. Therefore, we just need to show that $\xi$ satisfies condition $(b)$ of Definition \ref{psi}. To do so, without loss of generality we assume that $G$ is a non-complete connected graph. Therefore, we have $d(G)=\diam(G)$. For convenience, we set $d=d(G)$. Notice that $d\geq 2$, since $G$ is not a complete graph.
\par First we show that there exists a non-free vertex  $v\in [n]$ such that $\xi(G-v)\geq \xi(G)$. It is clear that for any non-free vertex $v$, we have $f(G-v)\geq f(G)$. If there exists a non-free vertex $v\in [n]$ such that $v$ does not belong to the vertex set of some LSP in $G$, then $v$ is certainly a vertex with the desired property. So, we may assume that every non-free vertex of $G$ belongs to the vertex set of every LSP in $G$. Let $P\hspace{-1.2mm}:v_1,v_2, \ldots ,v_{d+1}$ be an arbitrary LSP in $G$. Notice that $v_i$ is a non-free vertex of $G$ for every $2\leq i \leq d$. Also, every vertex $v$ of $G$ which is not on the path $P$ is a free  vertex of $G$. These imply that for every $2\leq j \leq d$, the vertices $v_{j-1}$ and $v_{j+1}$ belong to two different connected components of the graph $G-v_j$. Indeed, assume on contrary that there exists a path $P'\hspace{-1.2mm}:v_{j-1},u_1,\ldots,u_t,v_{j+1}$ in the graph $G-v_j$. Without loss of generality we may assume that $P'$ is an induced path in $G-v_j$. This clearly implies that $\{u_1,\ldots,u_t\}\subseteq V(P)\setminus \{v_j\}$. So, we have that  $u_1=v_{j-2},u_2=v_{j-3},\ldots,u_t=v_{j-t-1}$. Thus, $\{v_{j-t-1},v_{j+1}\}\in E(G)$, a contradiction.
\par Now, let $2\leq j \leq d$, and $\{n_j,v_j\}\in E(G)$, where $n_j\in [n]\setminus V(P)$. By a similar argument one could see that $d_{G-v_{j+1}}(n_j,v_1)=j$, if $\{n_j,v_{j-1}\}\notin E(G)$, and $d_{G-v_j}(n_j,v_1)=j-1$, if $\{n_j,v_{j-1}\}\in E(G)$. Also, we have $d_{G-v_{j-1}}(n_j,v_{d+1})=d-j+2$, if $\{n_j,v_{j+1}\}\notin E(G)$, and $d_{G-v_j}(n_j,v_{d+1})=d-j+1$, if $\{n_j,v_{j+1}\}\in E(G)$. We use these facts throughout the proof. 
\par Now, assume that there exists $2\leq i\leq d$ such that $N_G(v_i)\cap N_G(v_{i+1})\neq\emptyset$. Let $j=\min \{i: 2\leq i\leq d, N_G(v_i)\cap N_G(v_{i+1})\neq\emptyset\}.$ We claim that $v_j$ is a vertex with the desired property. To prove the claim, we distinguish the following cases:
\par First assume that $j>3$. Let $n_j\in N_G(v_j)\cap N_G(v_{j+1})$. We have $d_{G-v_j}(n_j,v_{d+1})=d-j+1$. Now, if $v_{j-1}$ is a free vertex of $G-v_j$, we have $f(G-v_j)\geq f(G)+1$. On the other hand, we have that $d(G-v_j)\geq j-2+d_{G-v_j}(n_j,v_{d+1})=d-1$, since the vertices $v_{j-1}$ and $v_{j+1}$ are contained in two connected components of the graph $G-v_j$. So, we get $f(G-v_j)+d(G-v_j)\geq f(G)+1+d-1=f(G)+d$, as desired. Now, if $v_{j-1}$ is a non-free vertex of $G-v_j$, then,  there exists $n_{j-1}\in N_{G-v_j}(v_{j-1})$ such that $n_{j-1}$ is not on the path $P$. Since  $j-2\geq 2$, by the choice of $j$, we have that $\{n_{j-1},v_{j-2}\}\notin E(G)$. This implies that $d_{G-v_j}(v_1,n_{j-1})=j-1$, and hence $d(G-v_j)\geq j-1+d_{G-v_j}(n_j,v_{d+1})=d$, which proves the claim in this case.
\par Next suppose that $j=3$. Let $n_3\in N_G(v_3)\cap N_G(v_{4})$. If  $v_2$ is a free vertex of $G-v_3$, then the claim follows by the same argument of the previous case. If $v_2$ is a non-free vertex of $G-v_3$, there exists  an induced path $P''\hspace{-1.2mm}:\alpha,v_2,\beta$ in $G-v_3$. Therefore $d(G-v_3)\geq 2+d_{G-v_3}(n_3,v_{d+1})=d$. This implies the claim.
\par Finally, suppose that $j=2$, and let $n_2\in N_G(v_2)\cap N_G(v_{3})$. Then $d_{G-v_2}(n_2,v_{d+1})=d-1$. Therefore, since $v_1$ and $v_3$ belong to two connected components of $G-v_2$, we get $d(G-v_2)\geq 1+d-1=d$, as desired.
\par Now, we assume that $N_G(v_i)\cap N_G(v_{i+1})=\emptyset$, for every $2\leq i\leq d$. We show that $v_2$ is a vertex with the desired property. First, notice that either $v_3$ is a free vertex of $G-v_2$, or there exists a vertex $n_3$ such that $\{n_3,v_3\}\in E(G)$ and $\{n_3,v_4\}\notin E(G)$.  This implies that either $f(G-v_2)\geq f(G)+1$ which completes the proof in this case or $d_{G-v_2}(n_3,v_{d+1})=d-1$.
\par \vspace{.5cm}
Now, for the rest of the proof, we show that all non-free vertices of $G$ satisfy conditions $(2)$ and $(3)$ of Definition \ref{psi}. Suppose that $v$ is an arbitrary non-free vertex of $G$. 
\par First we prove that $\xi(G_v)\geq \xi(G)$. Notice that $f(G_v)\geq f(G)+1$, by Lemma~\ref{simplicial}. Therefore, the result follows if we show that $d(G_v)\geq d-1$. Let $\alpha$ and $\beta$ be two vertices of $G$ with $d_G(\alpha,\beta)=d$. It suffices to show that $d_{G_v}(\alpha,\beta)\geq d-1$. Assume on contrary that there exists a path $P\hspace{-1.2mm}:\alpha=u_1,u_2, \ldots ,u_{\ell+1}=\beta$ in $G_v$ with $\ell\leq d-2$. Without loss of generality we may assume that $P$ is an induced path in $G_v$. Now we consider the following cases:
\par First assume that $v\in V(P)$. This clearly implies that $v=\alpha$ or $v=\beta$, since $v$ is a free vertex of $G_v$. Therefore, $P$ is a path in $G$. So, we have that $d_G(\alpha,\beta)\leq d-2$, a contradiction.
\par Next assume that $v\notin V(P)$. Now, since $P$ is an induced path in $G_v$ and since $P$ is not a path in $G$, we get $|N_{G_v}(v)\cap V(P)|=2$. Therefore, $N_{G_v}(v)\cap V(P)=\{u_i,u_{i+1}\}$ for some  $1\leq i \leq \ell$. Now, the path $P'\hspace{-1.2mm}:\alpha=u_1,\ldots,u_i,v,u_{i+1},\ldots,u_{\ell+1}=\beta$ is a path in $G$ between $\alpha$ and $\beta$ with the length at most $d-1$, which is a contradiction.
\par So, we have $d_{G_v}(\alpha,\beta)\geq d-1$, as desired.
\par \vspace{.5cm} Finally, we show that $\xi(G_{v}-v)\geq \xi(G)-1$. To prove this, it is enough to show that $d(G_{v}-v)\geq d-1$, since  $f(G_{v}-v)\geq f(G)$ by Lemma \ref{simplicial}. Let $P\hspace{-1.2mm}:v_1,v_2, \ldots ,v_{d+1}$ be an LSP in $G$. We consider the following cases:
\par First assume that $v \neq v_1$ and $v\neq v_{d+1}$. We show that  $d_{G_{v}-v}(v_1,v_{d+1})\geq d-1$. Suppose on contrary that 
$d_{G_{v}-v}(v_1,v_{d+1})\leq d-2$. Therefore, there exists a path $P'$ in $G_{v}-v$ between the vertices $v_1$ and $v_{d+1}$ with the length at most $d-2$. We may also assume that $P'$ is an induced path in $G_{v}-v$. Since $v$ is a free vertex of $G_v$ and since $P'$ is not a path in $G$, the vertex $v$ has exactly two adjacent neighbours in $G$ on the path $P'$. This implies that $d_G(v_1,v_{d+1})\leq d-1$, a contradiction. 
\par Next without loss of generality we assume that $v=v_1$. Now, the result follows if we show that $d_{G_{v}-v}(v_2,v_{d+1})\geq d-1$. Assume on contrary that there exists an induced path $P''\hspace{-1.2mm}:v_2,u_1,\ldots,u_r=v_{d+1}$ in $G_{v}-v$ between the vertices $v_2$ and $v_{d+1}$ with the length $r$, where $r\leq d-2$. Now, we have either $N_{G_v}(v)\cap V(P'')=\{v_2\}$ or $N_{G_v}(v)\cap V(P'')=\{v_2,u_1\}$. Therefore, we get a path in $G$ between $v_1$ and $v_{d+1}$, with the  length at most $d-1$, and hence $d_G(v_1,v_{d+1})\leq d-1$, a contradiction.
\par So, we have  $d(G_{v}-v)\geq d-1$, as desired.

\end{proof}
Now, combining of Theorem \ref{dcompatible} and Theorem \ref{xi}, we get the following combinatorial lower bound for the depth of binomial edge ideals.
\begin{Theorem}\label{mixed}
Let $G$ be a graph on $[n]$. Then
$$\depth S/J_G\geq \xi(G).$$
In particular, if $G$ is connected, then 
$$\depth S/J_G\geq f(G)+\diam(G).$$
\end{Theorem}

We would like to remark that the lower bound in Theorem \ref{mixed} could be tight. For instance, let $G$ be the graph illustrated in Figure \ref{sharp}. Then $f(G)=8$ and $\diam(G)=4$, and hence $\depth S/J_G\geq \xi(G)=12$, by Theorem~\ref{mixed}. On the other hand, the upper bound given in \eqref{upper} implies that $\depth S/J_G\leq |V(G)|+1=12$, since $\kappa(G)=1$. Therefore, we get $\depth S/J_G=12$. 
\begin{figure}[H]
\definecolor{qqqqff}{rgb}{0.,0.,1.}
\centering
\begin{tikzpicture}[scale=1.1,line cap=round,line join=round,>=triangle 45,x=1.0cm,y=1.0cm]
\draw (0.,0.)-- (1.,1.);
\draw (0.,0.)-- (1.,-1.);
\draw (1.,1.)-- (1.,0.);
\draw (1.,-1.)-- (1.,0.);
\draw (1.,0.)-- (2.,1.);
\draw (1.,0.)-- (2.,-1.);
\draw (2.,1.)-- (2.,-1.);
\draw (0.,0.)-- (-1.,1.);
\draw (0.,0.)-- (-1.,-1.);
\draw (-1.,1.)-- (-1.,0.);
\draw (-1.,-1.)-- (-1.,0.);
\draw (-2.,1.)-- (-1.,0.);
\draw (-1.,0.)-- (-2.,-1.);
\draw (-1.,0.)-- (0.,0.);
\draw (0.,0.)-- (1.,0.);
\draw (-2.,1.)-- (-2.,-1.);
\begin{scriptsize}
\draw [fill=qqqqff] (0.,0.) circle (1.5pt);
\draw [fill=qqqqff] (1.,1.) circle (1.5pt);
\draw [fill=qqqqff] (1.,0.) circle (1.5pt);
\draw [fill=qqqqff] (1.,-1.) circle (1.5pt);
\draw [fill=qqqqff] (2.,1.) circle (1.5pt);
\draw [fill=qqqqff] (2.,-1.) circle (1.5pt);
\draw [fill=qqqqff] (-1.,1.) circle (1.5pt);
\draw [fill=qqqqff] (-1.,0.) circle (1.5pt);
\draw [fill=qqqqff] (-1.,-1.) circle (1.5pt);
\draw [fill=qqqqff] (-2.,1.) circle (1.5pt);
\draw [fill=qqqqff] (-2.,-1.) circle (1.5pt);
\end{scriptsize}
\end{tikzpicture}
\vspace{1mm}
\caption{A graph $G$ with $\depth S/J_G=\xi(G)=n-\kappa(G)+2=12$.}
\label{sharp}
\end{figure}
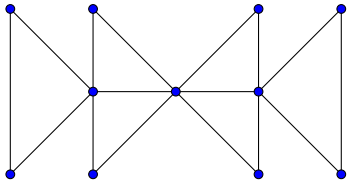

\par It is also remarkable that there are connected graphs $G$ on $[n]$ for which the lower bound given in Theorem \ref{mixed} is attained while $\depth S/J_G <n-\kappa(G)+2$. For instance, let $G$ be the graph depicted in Figure \ref{new} with $f(G)=\diam(G)=4$. Then, a routine computation with Macaulay2~\cite{GS} over the base field $\QQ$ shows that $\depth S/J_G=\xi(G)=8$, whereas $n-\kappa(G)+2=9$.

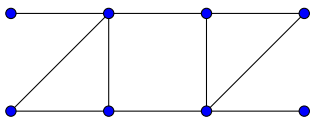
\begin{figure}[H]
\definecolor{qqqqff}{rgb}{0.,0.,1.}
\centering
\begin{tikzpicture}[scale=1.3,line cap=round,line join=round,>=triangle 45,x=1.0cm,y=1.0cm]
\draw (-2.,1.)-- (-2.,0.);
\draw (-2.,0.)-- (-3.,0.);
\draw (-2.,1.)-- (-1.,1.);
\draw (-1.,0.)-- (-2.,0.);
\draw (-2.,1.)-- (-3.,0.);
\draw (-2.,1.)-- (-3.,1.);
\draw (-1.,1.)-- (-1.,0.);
\draw (0.,1.)-- (-1.,1.);
\draw (0.,1.)-- (-1.,0.);
\draw (-1.,0.)-- (0.,0.);
\begin{scriptsize}
\draw [fill=qqqqff] (-2.,1.) circle (1.5pt);
\draw [fill=qqqqff] (-1.,1.) circle (1.5pt);
\draw [fill=qqqqff] (-2.,0.) circle (1.5pt);
\draw [fill=qqqqff] (-3.,0.) circle (1.5pt);
\draw [fill=qqqqff] (-1.,0.) circle (1.5pt);
\draw [fill=qqqqff] (0.,1.) circle (1.5pt);
\draw [fill=qqqqff] (-3.,1.) circle (1.5pt);
\draw [fill=qqqqff] (0.,0.) circle (1.5pt);
\end{scriptsize}
\end{tikzpicture}
\vspace{1.5mm}
\caption{A graph $G$ with $\depth S/J_G=\xi(G)=8<n-\kappa(G)+2=9$.}
\label{new}
\end{figure}

\par We should also remark that there are both Cohen-Macaulay and non Cohen-Macaulay graphs $G$ with $\depth S/J_G>\xi(G)$.  For example, for the illustrated graph $G$ in Figure \ref{CM} with $\xi(G)=5$, a computation with Macaulay2 \cite{GS} (see also \cite[Theorem~3.9]{KumarS}), shows that $S/J_G$ is Cohen-Macaulay, which yields that $\depth S/J_G=~6$. On the other hand, for the non Cohen-Macaulay graph $G$ depicted in Figure \ref{nice}, we have $\depth S/J_G=5>\xi(G)=2$, see Theorem \ref{main}. Moreover, there are Cohen-Macaulay and non Cohen-Macaulay graphs $G$ with $\depth S/J_G=\xi(G)$. Indeed, one could easily see that all graphs $G$ satisfying equivalent conditions of \cite[Theorem~3.1]{EHH} are Cohen-Macaulay with $\depth S/J_G=\xi(G)$. Also, the graph $G$ depicted in Figure~\ref{new} is non Cohen-Macaulay with $\depth S/J_G=\xi(G)=8$.
\begin{figure}[H]
\definecolor{qqqqff}{rgb}{0.,0.,1.}
\centering
\begin{tikzpicture}[scale=1.3,line cap=round,line join=round,>=triangle 45,x=1.0cm,y=1.0cm]
\draw (3.,2.)-- (2.,2.);
\draw (3.,2.)-- (4.,2.);
\draw (3.,2.)-- (4.,3.);
\draw (3.,2.)-- (3.,3.);
\draw (4.,2.)-- (4.,3.);
\draw (4.,3.)-- (3.,3.);
\begin{scriptsize}
\draw [fill=qqqqff] (2.,2.) circle (1.5pt);
\draw [fill=qqqqff] (3.,2.) circle (1.5pt);
\draw [fill=qqqqff] (4.,2.) circle (1.5pt);
\draw [fill=qqqqff] (4.,3.) circle (1.5pt);
\draw [fill=qqqqff] (3.,3.) circle (1.5pt);
\end{scriptsize}
\end{tikzpicture}
\vspace{1.7mm}
\caption{A Cohen-Macaulay graph $G$ with $\depth S/J_G=6>\xi(G)=5$.}
\label{CM}
\end{figure}
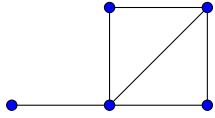

\par Now, combining Theorem \ref{mixed} together with \cite[Theorem~5.2]{RSK2} yields the following bound for the depth of binomial edge ideals.
\begin{Corollary}
Let $G$ be a graph on $[n]$ with $n\geq 3$. Then
\[
\depth S/J_G\geq \max\{4,\xi(G)\}.
\]
\end{Corollary}

Moreover, in \cite[Theorem~5.3]{RSK2}, it was shown that for graphs $G$ with more than three vertices, $\depth S/J_G=4$ if and only if $G=G'\ast 2K_1$ for some graph $G'$. Therefore, we have:

\begin{Corollary}\label{combin}
Let $G$ be a graph on $[n]$ with $n\geq 4$ and $G\neq G'\ast 2K_1$ for every graph $G'$. Then 
\[
\depth S/J_G\geq \max\{5,\xi(G)\}.
\]
\end{Corollary}
As we saw in Figure \ref{sharp}, there are graphs $G$ for which $\depth S/J_G=\xi(G)>5$. There are also graphs $G$ for which $\depth S/J_G=5$ while $\xi(G)<5$. For instance, let $G$ be the graph shown in Figure \ref{nice}. We have $\xi(G)=2$. However, we will see in Theorem \ref{main} that $\depth S/J_G=5$. Our main goal in the rest of this paper is to characterize all graphs $G$ with $\depth S/J_G=5$.

\section{A poset associated with binomial edge ideals and a Hochster type formula}\label{hochster type}
In this section, continuing the topological approach from \cite{A} and \cite{RSK2}, we study the local cohomology modules of binomial edge ideals. To this end, we first associate a new poset, adapted to our needs, with binomial edge ideals as follows. 
\par Let $I$ be an ideal in the polynomial ring $S$ and $I=q_1\cap \cdots \cap q_t$ be an arbitrary decomposition for the ideal $I$. Then, by the poset $\mathcal{R}_I$, ordered by the reverse inclusion, we mean the poset of all possible sums of ideals in this decomposition, defined in \cite[Example~2.1]{ABZ}. Now, we know that $J_{G}=\bigcap\limits_{\substack{T\subseteq [n]}}P_T(G)$. We use $\mathcal{R}_{G}$, instead of $\mathcal{R}_{J_G}$, to denote the poset arose from the above decomposition of $J_G$.

\par Now, we define another poset associated with the binomial edge ideal of a graph $G$. 
\begin{Definition}\label{typo}
\em{Let $G$ be a graph on $[n]$. Associated with the decomposition $J_{G}=\bigcap\limits_{\substack{T\subseteq [n]}}P_T(G)$, we consider the poset $(\mathcal{M}_G,\preccurlyeq)$, ordered by reverse inclusion, which is made up of the following elements:
\begin{itemize}
\item 
the prime ideals in the poset $\mathcal{R}_{G}$,
\item 
the prime ideals in the posets $\mathcal{R}_{I}$, arising from the following type of decompositions:
\[
I=q_1\cap q_2 \cap \cdots \cap q_t \cap
(q_1+P_\emptyset(G)) \cap (q_2+P_\emptyset(G)) \cap \cdots \cap (q_t+P_\emptyset(G)),\] 
where $I$'s are the non-prime ideals in the poset $\mathcal{R}_{G}$ and 
$q_1, q_2,\ldots, q_t$ are the minimal prime ideals of $I$, and
\item
the prime ideals that we obtain repeatedly by this procedure every time that we find a non-prime ideal.
\end{itemize}
}
\end{Definition}
Note that using the non-minimal primary decomposition $J_{G}=\bigcap\limits_{\substack{T\subseteq [n]}}P_T(G)$ in Definition \ref{typo}, turns the poset $\mathcal{M}_G$ to be different from the posets $\mathcal{A}_G$ and $\mathcal{Q}_G$, considered by the authors in \cite{A} and \cite{RSK2}, respectively. We also notice that the significance of this new definition will be demonstrated in the proof of Theorem ~\ref{main} and Theorem \ref{corona}.  Furthermore, the following lemma, which is a direct consequence of \cite[Proposition~3.4]{RSK2}, guarantees that the process of the construction of the poset $\mathcal{M}_G$ terminates after a finite number of steps just like the construction process of the poset $\mathcal{A}_{G}$ as well as the poset $\mathcal{Q}_G$.

\begin{Lemma}\label{crucial}
Let $G$ be a graph on $[n]$. Then every element $q$ of the poset $\mathcal{M}_G$ is of the form $P_T(H)$, for some graph $H$ on $[n]$ and some $T\subseteq [n]$.
\end{Lemma}

Now, by applying Lemma \ref{crucial} and thanks to the flexibility for the decomposition of the ideals in \cite[Theorem~5.22]{ABZ}, the following Hochster type decomposition formula for the local cohomology modules of binomial edge ideals is established by using the same argument that was applied in the proof of \cite[Theorem~3.9]{A}. We first need to fix a notation before stating the formula.

\par Let $1_{{\mathcal{M}_G}}$ be a terminal element that we add to the poset $\mathcal{M}_G$. Then, recall that for every $q\in \mathcal{M}_G$, by the interval $(q,1_{\mathcal{M}_G})$, one means the subposet
\[
\{z\in \mathcal{M}_G:
 q\precneqq z\precneqq 1_{{\mathcal{M}_G}}\}.
\]

\par 
\begin{Theorem}{\em(see} \cite[Theorem~3.9]{A} {\em and} \cite[Theorem~3.6]{RSK2}{\em .)}\label{Hochster}
Let $G$ be a graph on $[n]$. Then we have the $\KK$-isomorphism
\[
H_{\mathfrak{m}}^{i}(S/J_G)\cong\bigoplus_{q \in \mathcal{M}_G} H_{\mathfrak{m}}^{d_q}(S/q)^{\oplus M_{i,q}},
\]
where $d_q=\dim S/q$ and $M_{i,q}=\dim_{\KK} \widetilde{H}^{i-d_q-1}((q,1_{{\mathcal{M}_G}});\KK)$.
\end{Theorem}

Note that the above theorem suggests an interesting and also a wonderful method to study the depth of binomial edge ideals. Indeed, instead of working directly with the minimal graded free resolution of $S/J_G$, which is almost an intractable task, we may elaborate the topological properties of the subposets $(q,1_{{\mathcal{M}_G}})$ of $\mathcal{M}_G$. In this approach, beside the algebraic tools, we also employ the topological tools.

\section{Combinatorial characterization of some binomial edge ideals in terms of their depth}\label{topology}
At the end of Section \ref{benefit}, we provided a lower bound for the depth of binomial edge ideals. We showed that for a graph $G$ with more than three vertices, $\depth S/J_G\geq \max\{5,\xi(G)\}$, if $G\neq G'\ast 2K_1$ for every graph $G'$. Now, our main goal in this section is to give a combinatorial characterization of graphs $G$ with $\depth S/J_G=5$. For this aim, we need to study topological properties of certain subposets of $\mathcal{M}_G$ to compute some local cohomology modules of $S/J_G$. 

\par First we need to state the following definition. 
\begin{Definition}\label{gpich}
\em{
Let $T\subseteq [n]$ with $|T|=n-2$. Associated with $T$, we introduce a family of graphs on $[n]$, denoted by $\mathcal{G}_T$, such that for each $G\in \mathcal{G}_T$, there exist two non-adjacent vertices $u$ and $w$ of $G$ with $u,w\in [n]\backslash T$, and three disjoint subsets of $T$, say  $V_0$, $V_1$ and $V_2$ with $V_1,V_2\neq \emptyset$ and $\bigcup_{i=0}^{2}V_i=T$, such that the following conditions hold:
\begin{enumerate}
\item[{(1)}] $N_G(u)=V_0\cup V_1$ and $N_G(w)=V_0\cup V_2$.
\item[{(2)}] $\{v_1,v_2\}\in E(G)$, for every $v_1\in V_1$ and every $v_2\in V_2$.
\end{enumerate}
}  
\end{Definition}
\begin{Remark}
\em{Given three arbitrary graphs $G_0$, $G_1$ and $G_2$ on disjoint sets  of vertices $V_0$, $V_1$ and $V_2$, respectively, where $V_1,V_2\neq \emptyset$, we can construct a graph in the family $\mathcal{G}_T$ with  $T=\bigcup_{i=0}^{2}V_i$. Note that the vertices in $V_0$ can be adjacent to some vertices in $V_1$ and $V_2$.}
\end{Remark}
\par An explicit example of a graph $G$ for which $G\in \mathcal{G}_T$ for some $T\subseteq V(G)$ with $|T|=|V(G)|-2$ is depicted in Figure \ref{nice}.
\par Before stating the main theorem of this section we need to  introduce a family of graphs that is essential in our characterization. In the following, $3K_1$ denotes the graph consisting of three isolated vertices.

\begin{Definition}\label{D5}
\em{Let $G$ be a graph on $[n]$ with $G\neq G'\ast 2K_1$ for any graph $G'$. We say that $G$ is a $D_5$-type graph, if one of the following conditions holds:
\begin{enumerate}
\item[{(1)}] $G\in \mathcal{G}_T$ for some $T\subseteq [n]$;
\item[{(2)}] $G=H\ast 3K_1$, for some graph $H$;
\item[{(3)}] $G=H\ast(K_1\dot{\cup}K_2)$, for some graph $H$.
\end{enumerate}
}  
\end{Definition}

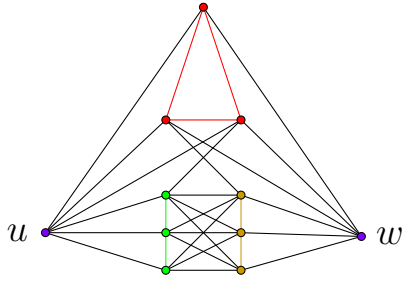
\begin{figure}[H]
\definecolor{qqffqq}{rgb}{0.,1.,0.}
\definecolor{cczzqq}{rgb}{0.8,0.6,0.}
\definecolor{ffqqqq}{rgb}{1.,0.,0.}
\definecolor{xfqqff}{rgb}{0.4980392156862745,0.,1.}
\centering
\begin{tikzpicture}[scale=1,line cap=round,line join=round,>=triangle 45,x=1.0cm,y=1.0cm]
\draw (-3.6,0.5)-- (-2.,2.);
\draw (-3.6,0.5)-- (-1.,2.);
\draw (-1.,1.)-- (-2.,1.);
\draw (-1.,0.)-- (-2.,0.);
\draw (-1.,1.)-- (-2.,0.);
\draw (-2.,1.)-- (-1.,0.);
\draw (-3.6,0.5)-- (-2.,0.);
\draw (0.6,0.45)-- (-1.,2.);
\draw (0.6,0.45)-- (-2.,2.);
\draw (0.6,0.45)-- (-1.,0.);
\draw (-2.,2.)-- (-1.,1.);
\draw (-1.,2.)-- (-2.,1.);
\draw [color=ffqqqq] (-1.5,3.5)-- (-1.,2.);
\draw [color=ffqqqq] (-1.5,3.5)-- (-2.,2.);
\draw (-1.5,3.5)-- (-3.6,0.5);
\draw (-1.5,3.5)-- (0.6,0.45);
\draw [color=ffqqqq] (-2.,2.)-- (-1.,2.);
\draw (-3.6,0.5)-- (-2.,1.);
\draw (0.6,0.45)-- (-1.,1.);
\draw (-2.,0.5)-- (-3.6,0.5);
\draw [color=qqffqq] (-2.,0.5)-- (-2.,1.);
\draw [color=qqffqq] (-2.,0.5)-- (-2.,0.);
\draw [color=cczzqq] (-1.,0.5)-- (-1.,0.);
\draw [color=cczzqq] (-1.,0.5)-- (-1.,1.);
\draw (0.6,0.45)-- (-1.,0.5);
\draw (-2.,1.)-- (-1.,0.5);
\draw (-1.,0.5)-- (-2.,0.);
\draw (-2.,0.5)-- (-1.,1.);
\draw (-4.26,0.75) node[anchor=north west] {\large{$u$}};
\draw (0.65,0.71) node[anchor=north west] {\large{$w$}};
\draw (-2.,0.5)-- (-1.,0.5);
\draw (-2.,0.5)-- (-1.,0.);
\begin{scriptsize}
\draw [fill=xfqqff] (-3.6,0.5) circle (1.5pt);
\draw [fill=ffqqqq] (-2.,2.) circle (1.5pt);
\draw [fill=ffqqqq] (-1.,2.) circle (1.5pt);
\draw [fill=cczzqq] (-1.,1.) circle (1.5pt);
\draw [fill=qqffqq] (-2.,1.) circle (1.5pt);
\draw [fill=cczzqq] (-1.,0.) circle (1.5pt);
\draw [fill=qqffqq] (-2.,0.) circle (1.5pt);
\draw [fill=xfqqff] (0.6,0.45) circle (1.5pt);
\draw [fill=ffqqqq] (-1.5,3.5) circle (1.5pt);
\draw [fill=qqffqq] (-2.,0.5) circle (1.5pt);
\draw [fill=cczzqq] (-1.,0.5) circle (1.5pt);
\end{scriptsize}
\end{tikzpicture}
\vspace{1mm}
\caption{A $D_5$-type graph $G$ with $G\in \mathcal{G}_T$, where  $T=V(G)\setminus\{u,w\}$.}
\label{nice}
\end{figure}

Now, we are ready to state the main result of this section which is an explicit characterization of graphs $G$ with $\depth S/J_G=5$.

\begin{Theorem}\label{main}
Let $G$ be a graph on $[n]$ with $n\geq 5$. Then the following statements are equivalent:
\begin{enumerate}
\item[{(a)}] $\depth S/J_G=5$.
\item[{(b)}] $G$ is a $D_5$-type graph.
\end{enumerate}  
\end{Theorem}

To prove the above theorem, we need to prepare several auxiliary ingredients. First, we state the following lemma that follows with the same argument as in the proof of \cite[Lemma~4.1]{RSK2}.

\begin{Lemma}\label{vital}
Let $G$ be a graph on $[n]$. Then $q+P_\emptyset(G)\in \mathcal{M}_G$, for every $q\in \mathcal{M}_G$.
\end{Lemma}	

We also need to recall a concept from the literature of topology of 
posets.
\begin{Definition}
\em{A poset $\mathcal{P}$ is said to be meet-contractible if there exists an element $\alpha\in \mathcal{P}$ such that $\alpha$ has a meet with every element $\beta\in \mathcal{P}$.}
\end{Definition}
The following lemma clarifies the importance of the notion of meet-contractible posets.
\begin{Lemma}\em{(\cite[Theorem~3.2]{BW}, see also \cite[Proposition~2.4]{W})}\label{meet}
\em{Every meet-contractible poset is contractible.}
\end{Lemma}

In the following theorem that is crucial in the proof of Theorem~\ref{main}, we discuss the vanishing of the zeroth and the first reduced cohomology groups of the subposets associated with the elements of $\mathcal{M}_G$, which are of the form $P_T(H)$ for some graph $H$ on $[n]$ and some $T\subseteq [n]$ with $|T|=n-2$.

\begin{Theorem}\label{corona}
Let $G$ be a graph on $[n]$ and $q\in \mathcal{M}_G$, where $q=P_{T}(H)$ for some graph $H$ on $[n]$ and some $T\subseteq [n]$ with $|T|=n-2$.
\begin{enumerate}
\item[{(a)}] If $c_H(T)=2$, then $(q,1_{\mathcal{M}_G})$ is connected if and only if $G\notin \mathcal{G}_T$.
\item[{(b)}] If $c_H(T)=1$, then $\widetilde{H}^{1}((q,1_{{\mathcal{M}_G}});\KK)=0$ if and only if $G\notin \mathcal{G}_{T}$.
\end{enumerate}  
\end{Theorem}

\begin{proof} 
Without loss of generality we assume that $T=\{1,\ldots,n-2\}$. \par $(a)$ We have $q=(x_1,\ldots,x_{n-2},y_1,\ldots,y_{n-2})$, since $c_H(T)=2$. Also, $\{n-1,n\}\notin E(G)$. Indeed, assume on  contrary that $\{n-1,n\}\in E(G)$. Since $q\in \mathcal{M}_G$, there exists $U\in \mathcal{C}(G)$ such that $P_U(G)\subseteq q$. It follows that $U\subseteq T$. On the other hand, $n-1$ and $n$ are two adjacent vertices of $G-U$. So, we get $f_{n-1,n}\in P_U(G)\subseteq q$, a contradiction.
\par Now, assume that $G\notin \mathcal{G}_T$. We show that $(q,1_{\mathcal{M}_G})$ is a connected poset. We proceed in the following steps:
\par Let $L_0=N_G(n-1)\cap N_G(n)$, $L_1=N_G(n)\backslash N_G(n-1)$, and $L_2=N_G(n-1)\backslash N_G(n)$. We also let $L_3=\{i\in T: i\notin N_G(n-1)\cup N_G(n)\}$. Set $X=\{P_{T\setminus \{\alpha\}}(G): \alpha \in L_1 \cup L_2 \cup L_3\}$. One has $X\subseteq(q,1_{\mathcal{M}_G})$, since $\{n-1,n\}\notin E(G)$ and $\alpha\notin L_0$ for every $\alpha\in L_1\cup L_2\cup L_3$. 
\par \textbf{Step 1:} Let $q'\in (q,1_{\mathcal{M}_G})$. We claim that there exists $P_{T\setminus\{\alpha\}}(G)\in X$ such that there is a path between $q'$ and $P_{T\setminus\{\alpha\}}(G)$ in the 1-skeleton graph of $(q,1_{\mathcal{M}_G})$. 
\par By Lemma \ref{crucial} we have that $q'=P_{T'}(H')$, for some graph $H'$ on $[n]$ and some $T'\subseteq [n]$. Now, there exists $U\in \mathcal{C}(G)$ such that $P_U(G)\subseteq q'$. It follows that $U\subseteq T'\subsetneq T$, since $q'\in (q,1_{\mathcal{M}_G})$. Now we consider the following cases:
\par First assume that $T\setminus U\subseteq L_3$. Therefore, the vertices $n-1$ and $n$ are isolated in $G-U$. This implies that  $P_U(G)\subseteq P_{T\setminus\{\alpha\}}(G)$, for every $\alpha\in T\setminus U$.
\par Next assume that $T\setminus U\nsubseteq L_3$. Let $\alpha\in (T\setminus U)\setminus L_3$. Clearly, $\alpha\notin L_0$. Indeed, otherwise we get $f_{n-1,n}\in P_U(G)$, a contradiction. So, without loss of generality we assume that $\alpha\in L_1$. It follows that there is no path between vertices $\alpha$ and $n-1$ in $G-U$. Therefore, we have that   $P_U(G)\subseteq P_{T\setminus\{\alpha\}}(G)$.
\par Thus, the claim follows from both above cases.
\par \textbf{Step 2:} Assume that $L_3\neq \emptyset$. Let $\alpha\in L_3$ and $q'\in (q,1_{\mathcal{M}_G})$. We show that there exists a path between $q'$ and $P_{T\setminus\{\alpha\}}(G)$ in the 1-skeleton graph of $(q,1_{\mathcal{M}_G})$.
\par By Step 1, there exists $P_{T\setminus\{\beta\}}(G)\in X$ such that there is a path between $q'$ and $P_{T\setminus\{\beta\}}(G)$ in the 1-skeleton graph of $(q,1_{\mathcal{M}_G})$. Moreover, since $\alpha\in L_3$, it is not difficult to see that $P_{T\setminus\{\alpha\}}(G)+P_{T\setminus\{\alpha,\beta\}}(G)\in (q,1_{\mathcal{M}_G})$ and $P_{T\setminus\{\beta\}}(G)\supseteq P_{T\setminus\{\alpha,\beta\}}(G)$. Therefore, we have the path
\[P_{T\setminus\{\alpha\}}(G),P_{T\setminus\{\alpha\}}(G)+P_{T\setminus\{\alpha,\beta\}}(G),P_{T\setminus\{\alpha,\beta\}}(G),P_{T\setminus\{\beta\}}(G),\]
in the 1-skeleton graph of $(q,1_{\mathcal{M}_G})$. This implies that 
$(q,1_{\mathcal{M}_G})$ is connected. So, for the rest of the proof we may assume that $L_3=\emptyset$.
\par \textbf{Step 3:} Now assume that $\alpha,\beta\in L_1$, and $\alpha<\beta$. We claim that there exists a path in the 1-skeleton graph of $(q,1_{\mathcal{M}_G})$ between $P_{T\setminus\{\alpha\}}(G)$ and $P_{T\setminus\{\beta\}}(G)$, (the situation in the case $\alpha,\beta\in L_2$ is similar).
\par We have that
\[P_{T\setminus\{\alpha\}}(G)=(x_i,y_i: i\in T\setminus\{\alpha\})+(f_{\alpha,n}),\]
and
\[P_{T\setminus\{\beta\}}(G)=(x_i,y_i: i\in T\setminus\{\beta\})+(f_{\beta,n}).\]
So, we get
\[P_{T\setminus\{\alpha,\beta\}}(G)=(x_i,y_i: i\in T\setminus\{\alpha,\beta\})+(f_{\alpha,n},f_{\beta,n},f_{\alpha,\beta}).\]
Therefore, we get the path
\[P_{T\setminus\{\alpha\}}(G),P_{T\setminus\{\alpha,\beta\}}(G),P_{T\setminus\{\beta\}}(G)\]
in the 1-skeleton graph of $(q,1_{\mathcal{M}_G})$, as desired.
\par \textbf{Step 4:} Let $\alpha \in L_1$ and $\beta \in L_2$. We show that there exists a path in the 1-skeleton graph of $(q,1_{\mathcal{M}_G})$ between the vertices $P_{T\setminus\{\alpha\}}(G)$ and $P_{T\setminus\{\beta\}}(G)$.
\par First assume that $\{\alpha,\beta\}\notin E(G)$. It follows that 
\[P_{T\setminus\{\alpha,\beta\}}(G)=(x_i,y_i: i\in T\setminus\{\alpha,\beta\})+(f_{\beta,n-1},f_{\alpha,n}).\]
Therefore, we get the path 
\[P_{T\setminus\{\alpha\}}(G),P_{T\setminus\{\alpha,\beta\}}(G),P_{T\setminus\{\beta\}}(G),\]
as desired.
\par Next assume that $\{\alpha,\beta\}\in E(G)$. Now, there exist vertices  $t_1\in L_1$ and $t_2\in L_2$, such that $\{t_1,t_2\}\notin E(G)$. Indeed, if no such vertices exist, then by putting $u=n$ and $w=n-1$ and also $V_0=L_0,V_1=L_1$ and $V_2=L_2$ in Definition \ref{gpich}, we get $G\in \mathcal{G}_T$, since $\{n-1,n\}\notin E(G)$ and $L_3=\emptyset$. Therefore, we get a contradiction. Now, by Step 3, there is a path between $P_{T\setminus\{\alpha\}}(G)$ and $P_{T\setminus\{t_1\}}(G)$, and also a path between $P_{T\setminus\{\beta\}}(G)$ and $P_{T\setminus\{t_2\}}(G)$, in the underlying graph of $(q,1_{\mathcal{M}_G})$. On the other hand,  since $\{t_1,t_2\}\notin E(G)$, the argument that we used in the first part of this step yields a path between $P_{T\setminus\{t_1\}}(G)$ and $P_{T\setminus\{t_2\}}(G)$. Therefore, we get a path between $P_{T\setminus\{\alpha\}}(G)$ and $P_{T\setminus\{\beta\}}(G)$ in the 1-skeleton graph of $(q,1_{\mathcal{M}_G})$,  as desired.
\par \textbf{Step 5:} Now suppose that $q_{1}, q_{2}\in (q,1_{\mathcal{M}_G})$ and $q_1\neq q_2$. Then, by Step 1, there exist $P_{T\setminus\{\alpha\}}(G),P_{T\setminus\{\beta\}}(G)\in X$ such that there is a path between $q_1$ and $P_{T\setminus\{\alpha\}}(G)$, and a path between $q_2$ and $P_{T\setminus\{\beta\}}(G)$ in the 1-skeleton graph of $(q,1_{\mathcal{M}_G})$. Moreover, we have a path between $P_{T\setminus\{\alpha\}}(G)$ and $P_{T\setminus\{\beta\}}(G)$, by Steps 3 and 4. Therefore, we get a path between $q_1$ and $q_2$ in the 1-skeleton graph of $(q,1_{\mathcal{M}_G})$. Thus, $(q,1_{\mathcal{M}_G})$ is connected.

\par \vspace{.5cm}

For the converse, assume that $G\in \mathcal{G}_T$. Therefore, there exist three disjoint subsets $V_0$, $V_1$ and $V_2$ of $T$, such that the conditions of Definition \ref{gpich} hold. Now, let $q_1=P_{V_0\cup V_1}(G)$, and $q_2=P_{V_0\cup V_2}(G)$. We have that $q_1,q_2\in (q,1_{\mathcal{M}_G})$. Now, we claim that there is no path between $q_1$ and $q_2$ in the 1-skeleton graph of $(q,1_{\mathcal{M}_G})$, and then we conclude the result.
\par Notice that $\{V_0\cup V_1,V_0\cup V_2\}\subseteq \mathcal{C}(G)$. Moreover, we have $T'\in \{V_0\cup V_1,V_0\cup V_2\}$, for every $T'\in \mathcal{C}(G)$ such that $n-1\notin T'$ and $n\notin T'$. Indeed, one could see that $T'\supseteq V_0\cup V_1$ or $T'\supseteq V_0\cup V_2$, for every $T'\in \mathcal{C}(G)$ such that $n-1\notin T'$ and $n\notin T'$. Without loss of generality we can assume that $T'\supseteq V_0\cup V_1$. Now we claim that $T'= V_0\cup V_1$. Assume on contrary that there exists  $v_2\in T'\cap V_2$. This implies that $v_2$ is not a cut vertex of $G-(T'\setminus \{v_2\})$, a contradiction to the fact that $T'\in \mathcal{C}(G)$. Therefore, we have either $q'\supseteq q_1$ or $q'\supseteq q_2$, for every $q'\in (q,1_{\mathcal{M}_G})$. Now, assume on contrary that there exists a path $\ell\hspace{-1.2mm}:q_1,q'_1,\ldots,q'_t,q_2$, between $q_1$ and $q_2$ in the 1-skeleton graph of $(q,1_{\mathcal{M}_G})$. Moreover, we may assume that $\ell$ is an induced path between the vertices $q_1$ and $q_2$. Now, we have that $t\geq 2$. Indeed, $t=1$ implies that $q'_1\supseteq q_1$ and  $q'_1\supseteq q_2$, and hence $q'_1=q$, a contradiction.
\par Now, if $q'_1\subseteq q'_2$, then $q'_2\supseteq q_1$. This clearly contradicts the minimality of the path $\ell$. So, we have $q'_1\supseteq q'_2$. On the other hand, we have $q'_2\supseteq q_2$, since $q'_2\nsupseteq q_1$. Therefore, $q'_1\supseteq q_2$, a contradiction to the minimality of $\ell$. 
\par So, there is no path between two vertices $q_1$ and $q_2$ in the 1-skeleton graph of $(q,1_{\mathcal{M}_G})$, as desired.

\par \vspace{.5cm}

$(b)$ Clearly we have $q=(x_1,\ldots,x_{n-2},y_1,\ldots,y_{n-2})+(f_{n-1,n})$, since $c_H(T)=1$. Let $q'=(x_1,\ldots,x_{n-2},y_1,\ldots,y_{n-2})$.
\par First assume that $G\notin \mathcal{G}_T$. We show that $\widetilde{H}^{1}((q,1_{{\mathcal{M}_G}});\KK)=0$. We consider the following cases:
\par Case 1: Assume that $q'\notin (q,1_{\mathcal{M}_G})$. We claim that $(q,1_{\mathcal{M}_G})$ is a meet-contractible poset. Then the result follows by Lemma \ref{meet}. 
\par Clearly, $P_\emptyset(G)\in (q,1_{\mathcal{M}_G})$. Now, let $q_1\in (q,1_{\mathcal{M}_G})$ such that $q_1\neq P_\emptyset(G)$. By Lemma \ref{crucial}, we have that $q_1=P_{T_1}(H_1)$, for some graph $H_1$ on $[n]$ and some $T_1\subseteq [n]$. Also, we have that $T_1\subsetneqq T$, since $q'\notin (q,1_{\mathcal{M}_G})$. Therefore, $q_1+P_\emptyset(G)\in (q,1_{\mathcal{M}_G})$, by Lemma \ref{vital}. Moreover, it is observed that $q_1+P_\emptyset(G)$ is the meet of the elements $q_1$ and $P_\emptyset(G)$. Therefore, $(q,1_{\mathcal{M}_G})$ is a meet-contractible poset, as desired.
\par Case 2: Assume that $q'\in (q,1_{\mathcal{M}_G})$. Let $\Delta=\Delta(q,1_{\mathcal{M}_G})$, $\Delta_1=\mathrm{star}_{\Delta}(q')$ and $\Delta_2=\mathrm{del}_{\Delta}(q')$. One has $\mathrm{del}_{\Delta}(q')=\Delta ((q,1_{\mathcal{M}_G})\setminus \{q'\})$. Then by a similar argument as in Case 1, it follows that $\Delta_2$ is a contractible simplicial complex. On the other hand $\Delta_1\cap\Delta_2=\mathrm{link}_\Delta(q')=\Delta(q',1_{\mathcal{M}_G})$, since $q'$ is a  minimal element in the poset $(q,1_{\mathcal{M}_G})$.

\par Now, first assume that $\Delta_1\cap\Delta_2=\{\emptyset\}$. So, clearly we have ${H}_{0}(\mathrm{link}_{\Delta}(q');\KK)=0$. Now, since $\Delta_1$ is a cone and $\Delta_2$ is contractible, the  Mayer-Vietoris sequence
\[
\cdots\rightarrow
H_1(\mathrm{star}_{\Delta}(q');\KK)\oplus H_1(\mathrm{del}_{\Delta}(q');\KK)\rightarrow H_1(\Delta;\KK)\rightarrow H_{0}(\mathrm{link}_{\Delta}(q');\KK)\rightarrow \cdots
\]
implies the result.
\par Next assume that $\Delta_1\cap\Delta_2\neq \{\emptyset\}$. Hence, the reduced Mayer-Vietoris sequence 
\[
\cdots \rightarrow \widetilde{H}_1(\mathrm{star}_{\Delta}(q');\KK)\oplus \widetilde{H}_1(\mathrm{del}_{\Delta}(q');\KK)\rightarrow \widetilde{H}_1(\Delta;\KK)\rightarrow \widetilde{H}_{0}(\mathrm{link}_{\Delta}(q');\KK)\rightarrow\cdots
\]
is induced. On the other hand, $q'\in\mathcal{M}_G$. Thus, By Lemma \ref{crucial}, we have $q'=P_{T'}(H')$, for some graph $H'$ on $[n]$ and some $T'\subseteq [n]$. Now, it is easily seen that  $T'=T$. Also, $c_{H'}(T')=2$, since otherwise we get $q'=q$, a contradiction.  Therefore, since $G\notin\mathcal{G}_T$, by part $(a)$ we have $\widetilde{H}_{0}(\mathrm{link}_{\Delta}(q');\KK)=\widetilde{H}_{0}(\Delta(q',1_{\mathcal{M}_G});\KK)=0$, and hence by using the latter Mayer-Vietoris sequence we get the result.

\par \vspace{.5cm}
 
Now, for the converse, assume that $G\in \mathcal{G}_T$. We show that $\widetilde{H}^{1}((q,1_{{\mathcal{M}_G}});\KK)\neq0$. It is clear by Definition \ref{gpich} that $q'=P_T(G)$, and hence $q'\in (q,1_{\mathcal{M}_G})$. Moreover, Definition \ref{gpich} again implies that $(q',1_{\mathcal{M}_G})$ is a non-empty poset. Indeed, by the notation of Definition \ref{gpich} we have $P_{V_0\cup V_1}(G),P_{V_0\cup V_2}(G)\in (q',1_{\mathcal{M}_G})$.  Now by keeping the same notation and also by the same argument as in Case 2 in above, the exact sequence
 
\[
0 \rightarrow \widetilde{H}_1(\Delta;\KK)\rightarrow \widetilde{H}_{0}(\mathrm{link}_{\Delta}(q');\KK)\rightarrow 0
\]
is induced. On the other hand, we have $c_G(T)=2$, by Definition \ref{gpich}. Therefore, since $G\in \mathcal{G}_T$, part $(a)$ implies that $\widetilde{H}_{0}(\mathrm{link}_{\Delta}(q');\KK)\neq0$, and hence the result follows. 

\end{proof}

Now we state the following lemma which is used in the proof of Theorem \ref{main}.  

\begin{Lemma}\label{nice suggestion}
Let $G$ be a graph on $[n]$ with $n\geq 4$. If there exists $T'\in \mathcal{C}(G)$ with $|T'|=n-3$ and $c_G(T')=2$, then either $G\in \mathcal{G}_{T}$ for some $T\subseteq [n]$, or $G=H\ast(K_1\dot{\cup}K_2)$ for some graph $H$ on $n-3$ vertices.

\end{Lemma}
\begin{proof}
Without loss of generality assume that $T'=\{1,\ldots,n-3\}$. We may also assume that $n$ is the isolated vertex of $G-T'$, since $c_G(T')=2$. Therefore, for every $i\in T'$ we have $i\in N_G(n)$, and either $i\in N_G(n-1)$ or $i\in N_G(n-2)$, since $T'\in \mathcal{C}(G)$. Now suppose that $G\neq H\ast(K_1\dot{\cup}K_2)$, for every graph $H$ on $n-3$ vertices. We show that $G\in \mathcal{G}_{T}$ for some $T\subseteq [n]$. 
\par Clearly, there exists $j\in \{n-2,n-1\}$ such that $T'\nsubseteq N_G(j)$. Let $V_0=N_G(j)\cap T', V_1=T'\setminus N_G(j)$ and $V_2=N_G(j)\setminus T'$. Now, by putting $u=n$ and $w=j$ in Definition \ref{gpich}, one checks that $G\in \mathcal{G}_T$, where $T=[n]\setminus \{n,j\}$.
\end{proof}

Finally, we need to state the following remark that will be used in the proof of Theorem~\ref{main}. The proof of this remark can be verified by taking a precise look at the proof of \cite[Theorem~4.4]{RSK2}.
\begin{Remark}\label{remark}
Let $G$ be a graph on $[n]$ with $n\geq 2$. Let $\widehat{\mathfrak{m}}=P_T(H)$, where $H$ is a graph on $[n]$ with $|T|=n-1$. Then the posets $(\mathfrak{m},1_{\mathcal{M}_G})$ and  $(\widehat{\mathfrak{m}},1_{\mathcal{M}_G})$ are contractible.
\end{Remark}
Now we are ready to prove the main theorem of this section.
\par \medskip Proof of Theorem \ref{main}: $(a)\Rightarrow (b)$ Assume that $G$ is not a $D_5$-type graph. We show $\depth S/J_G\neq 5$. Notice that if $G=G'\ast 2K_1$ for some graph $G'$, then $\depth S/J_G=4$, by \cite[Theorem~5.3]{RSK2}. So, we may assume that $G\neq G'\ast 2K_1$, for any graph $G'$. Now, by the definition of depth and by Theorem \ref{Hochster}, it is enough to show that $M_{5,q}=\dim_{\KK} \widetilde{H}^{4-d_q}((q,1_{{\mathcal{M}_G}});\KK)=0$, for all $q\in \mathcal{M}_G$. 
\par Let $q$ be an arbitrary element of the poset $\mathcal{M}_G$. We have $q=P_{T}(H)$ for some graph $H$ on $[n]$ and some $T\subseteq [n]$, by Lemma~\ref{crucial}. If $d_q\geq 6$, then obviously we have the result. Moreover, it is easily seen that there is no $q'\in \mathcal{M}_G$ such that $d_{q'}=1$. Therefore, we assume that $d_q\in\{0,2,3,4,5\}$. Now, we consider the following cases:
\par Let $d_q\in \{0,2\}$. So, $\mathrm{ht}\hspace{0.9mm}q\in \{2n-2,2n\}$. This implies that $|T|-c_H(T)\in\{n-2,n\}$. Therefore, without loss of generality, we assume that either $q=\mathfrak{m}$ or $q=\widehat{\mathfrak{m}}$. Now, the result follows, since the posets $(\mathfrak{m},1_{{\mathcal{M}_G}})$ and $(\widehat{\mathfrak{m}},1_{{\mathcal{M}_G}})$ are contractible by Remark \ref{remark}.
\par Let $d_q=3$. Then $|T|-c_H(T)=n-3$, and hence $|T|=n-2$ and $c_H(T)=1$. Now, by Definition \ref{D5}, $G\notin \mathcal{G}_T$, since $G$ is not a $D_5$-type graph. Thus, we have $M_{5,q}=\dim_{\KK}\widetilde{H}^{1}((q,1_{{\mathcal{M}_G}});\KK)=0$, by Theorem \ref{corona} part $(b)$.
\par Let $d_q=4$. It follows that $|T|-c_H(T)=n-4$. So, we have  either $|T|=n-2$ and $c_H(T)=2$, or $|T|=n-3$ and $c_H(T)=1$. 
\par First assume that $|T|=n-2$ and $c_H(T)=2$. Therefore, Theorem~\ref{corona} part $(a)$ implies that $M_{5,q}=\dim_{\KK} \widetilde{H}^{0}((q,1_{{\mathcal{M}_G}});\KK)=0$, since $G\notin \mathcal{G}_T$.
\par Next assume that $|T|=n-3$ and $c_H(T)=1$. Without loss of generality we assume that $T=\{1,\ldots,n-3\}$. Therefore,
\[q=(x_1,\ldots,x_{n-3},y_1,\ldots,y_{n-3})+(f_{n-2,n-1},f_{n-2,n},f_{n-1,n}).\]
\par Let $q_{1}, q_{2}\in (q,1_{\mathcal{M}_G})$ and $q_1\neq q_2$. Then, there exist $T_1,T_2\in \mathcal{C}(G)$ such that $q_1\supseteq P_{T_1}(G)$ and $q_2\supseteq P_{T_2}(G)$. Moreover, we have $T_1,T_2\subseteq T$, since $q_{1}, q_{2}\in (q,1_{\mathcal{M}_G})$. Now, we distinguish the following cases:
\par First assume that $T_1,T_2\subsetneqq T$. Therefore, by Lemma \ref{vital}, we have that $q_{1}+P_{\emptyset}(G)\in (q,1_{\mathcal{M}_G})$ and $q_{2}+P_{\emptyset}(G)\in (q,1_{\mathcal{M}_G})$, since $q_{1}+P_{\emptyset}(G)\subsetneqq q$ and $q_{2}+P_{\emptyset}(G)\subsetneqq q$. So, we get the path
\[q_{1}, P_{T_1}(G), P_{T_1}(G)+P_{\emptyset}(G), P_{\emptyset}(G),  P_{\emptyset}(G)+P_{T_2}(G), P_{T_2}(G), q_{2}\]
in the 1-skeleton graph of the order complex of the poset $(q,1_{\mathcal{M}_G})$.
\par Next assume that $T_1=T$ or $T_2=T$. If $T_1=T_2$, then we get the path $q_1,P_{T_1}(G)=P_{T_2}(G),q_2$, as desired. So, without loss of generality, we assume that $T_1=T$ and $T_2\subsetneqq T$. One has $c_G(T_1)\in\{2,3\}$, since $T_1\in \mathcal{C}(G)$. Notice that if $c_G(T_1)=2$, then Lemma \ref{nice suggestion} implies  that $G$ is a $D_5$-type graph, a contradiction. So, we have $c_G(T_1)=3$. Furthermore, there exist two vertices $\alpha\in T_1$ and $\beta\in\{n-2,n-1,n\}$, such that $\{\alpha,\beta\}\notin E(G)$. Indeed, otherwise we get $G=H\ast 3K_1$, where $H=G-\{n-2,n-1,n\}$, which is a contradiction with Definition \ref{D5}. Now, one could see that 
\[P_{T_1}(G)+P_{{T_1}\setminus\{\alpha\}}(G)=(x_1,\ldots,x_{n-3},y_1,\ldots,y_{n-3})+(f_{i,j}),\]
where $i<j$ and $i,j \in \{n-2,n-1,n\}\setminus \{\beta\}$. 
Therefore, we have that $P_{T_1}(G)+P_{{T_1}\setminus\{\alpha\}}(G)\in (q,1_{\mathcal{M}_G})$. So, we get the path
%\noindent \medskip
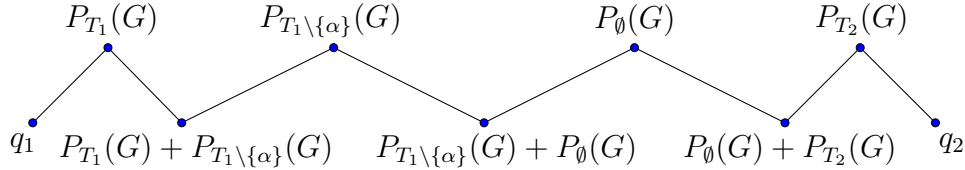
\begin{figure}[H]
\centering
\begin{tikzpicture}[scale=1,line cap=round,line join=round,>=triangle 45,x=1.0cm,y=1.0cm]
\draw (-5.,0.)-- (-4.,1.);
\draw (-4.,1.)-- (-3.02,0.);
\draw (-3.02,0.)-- (-1.,1.);
\draw (-1.,1.)-- (1.,0.);
\draw (1.,0.)-- (3.,1.);
\draw (3.,1.)-- (5.,0.);
\draw (5.,0.)-- (6.,1.);
\draw (6.,1.)-- (7.,0.);
\draw (-5.45,0.) node[anchor=north west] {$q_1$};
\draw (-4.72,1.73) node[anchor=north west] {$P_{T_1}(G)$};
\draw (-4.8,0.) node[anchor=north west] {$P_{T_1}(G)+P_{{T_1}\setminus\{\alpha\}}(G)$};
\draw (-2.13,1.73) node[anchor=north west] {$P_{{T_1}\setminus\{\alpha\}}(G)$};
\draw (-0.59,0.) node[anchor=north west] {$P_{{T_1}\setminus\{\alpha\}}(G)+P_{\emptyset}(G)$};
\draw (2.32,1.72) node[anchor=north west] {$P_{\emptyset}(G)$};
\draw (3.5,0.) node[anchor=north west] {$P_{\emptyset}(G)+P_{T_2}(G)$};
\draw (5.25,1.73) node[anchor=north west] {$P_{T_2}(G)$};
\draw (6.9,0.) node[anchor=north west] {$q_{2}$};
\begin{scriptsize}
\draw [fill=blue] (3.,1.) circle (1.5pt);
\draw [fill=blue] (5.,0.) circle (1.5pt);
\draw [fill=blue] (6.,1.) circle (1.5pt);
\draw [fill=blue] (7.,0.) circle (1.5pt);
\draw [fill=blue] (-1.,1.) circle (1.5pt);
\draw [fill=blue] (-3.02,0.) circle (1.5pt);
\draw [fill=blue] (-4.,1.) circle (1.5pt);
\draw [fill=blue] (-5.,0.) circle (1.5pt);
\draw [fill=blue] (1.,0.) circle (1.5pt);
\end{scriptsize}
\end{tikzpicture}
\end{figure}

between $q_1$ and $q_2$ in the 1-skeleton graph of the order complex of the poset $(q,1_{\mathcal{M}_G})$.
\par Therefore, it follows from the both cases that $(q,1_{\mathcal{M}_G})$ is connected, and hence the desired result follows.
\par Let $d_q=5$. It follows that $|T|-c_H(T)=n-5$. So, we have that either $|T|=n-3$ and $c_H(T)=2$, or $|T|=n-4$ and $c_H(T)=1$. Now the result follows once we show that $(q,1_{{\mathcal{M}_G}})$ is a non-empty poset.
\par First suppose that $|T|=n-3$ and $c_H(T)=2$. Assume on  contrary that $(q,1_{{\mathcal{M}_G}})$ is an empty poset, i.e. $q\in \Min(J_G)$. Then, $q=P_{T'}(G)$ for some $T'\in \mathcal{C}(G)$. It follows that $T'=T$ and $c_G(T')=c_H(T)=2$. Therefore, Lemma \ref{nice suggestion} implies that $G$ is a $D_5$-type graph, a contradiction.
\par Next suppose that $|T|=n-4$ and $c_H(T)=1$. Therefore, we have $P_{\emptyset}(G)\subseteq q$. Moreover, the assumption $n\geq 5$ implies that $T\neq \emptyset$, and hence $P_{\emptyset}(G)\in (q,1_{\mathcal{M}_G})$, as desired.
 
\par \vspace{.5cm}

$(b)\Rightarrow (a)$ Assume that $G$ is a $D_5$-type graph. Then, Corollary \ref{combin} implies that $\depth S/J_G\geq5$.  Therefore, the result follows if we show that $\depth S/J_G\leq5$. To do so, we consider the following cases:
\par First assume that $G\in \mathcal{G}_T$, for some $T\subseteq [n]$ with $|T|=n-2$. Let $q=P_T(G)$. We have $q\in \mathcal{M}_G$ and by Definition \ref{gpich} we have $c_{G}(T)=2$. Thus, Theorem \ref{corona} part $(a)$ implies that $(q,1_{\mathcal{M}_G})$ is not connected. Since $d_q=4$, we have that $M_{5,q}=\dim_{\KK} \widetilde{H}^{0}((q,1_{{\mathcal{M}_G}});\KK)\neq 0$. Therefore, we get $H_{\mathfrak{m}}^{5}(S/J_G)\neq 0$, by Theorem~ \ref{Hochster}. This yields that $\depth S/J_G\leq5$, as desired.
\par Next assume that $G=H\ast 3K_1$, for some graph $H$. If $H$ is a complete graph, then the result follows by \cite[Theorem~3.9]{KumarS}. If $H$ is not complete, then it follows from  \cite[Theorem~4.3]{KumarS} and \cite[Theorem~4.4]{KumarS} that $\depth S/J_G=5$, as desired.
\par Finally, suppose that $G=H\ast(K_1\dot{\cup}K_2)$, for some graph $H$. The result follows by \cite[Theorem~3.9]{KumarS} if $H$ is a complete graph. If $H$ is not complete, then it follows from  \cite[Theorem~4.3]{KumarS} and \cite[Theorem~4.4]{KumarS} that $\depth S/J_G=5$, as desired.

\vspace{1cm}
\par \textbf{Acknowledgments:} The authors would like to thank the anonymous referees for their careful reading of the manuscript and for their valuable comments and suggestions. The authors would also like to thank Institute for Research in Fundamental Sciences (IPM) for financial support. The research of the second author was in part supported by a grant from IPM (No. 1400130113).

\end{document}